\documentclass[10pt]{article}
\usepackage{mathrsfs}
\usepackage{amsthm}
\usepackage{amssymb}
\usepackage{enumerate}
\usepackage{amsmath}
\usepackage{graphicx}
\usepackage{amsfonts}
\usepackage{float}
\usepackage{cite}
\usepackage[text={150mm,220mm}]{geometry}
\usepackage[colorlinks,
            linkcolor=blue,
            anchorcolor=blue,
            citecolor=blue
            ]{hyperref}
\usepackage{xcolor}

\newtheorem{theorem}{Theorem}[section]

\newtheorem{lemma}[theorem]{Lemma}

\numberwithin{equation}{section}
\normalsize

\title{Hydrodynamics of the Binary Contact Path Process}
\author{Xiaofeng Xue \thanks{\textbf{E-mail}: xfxue@bjtu.edu.cn \textbf{Address}: School of Science, Beijing Jiaotong University, Beijing 100044, China.} and Linjie Zhao \thanks{\textbf{E-mail}: zhaolinjie@pku.edu.cn \textbf{Address}: School of Mathematical Sciences, Peking University, Beijing 100049, China.}}
\date{December 2018}

\begin{document}

\maketitle

\noindent {\bf Abstract:} In this paper we are concerned with the binary contact path process introduced in \cite{Gri1983} on the lattice $\mathbb{Z}^d$ with $d\geq 3$. Our main result gives a hydrodynamic limit of the process, which is the solution to a heat equation. The proof of our result follows the strategy introduced in \cite{kipnis+landim99} to give hydrodynamic limit of the SEP model with some details modified since the states of all vertices are not uniformly bounded for the binary contact path process.  In the modifications, the theory of the linear system introduced in \cite{Lig1985} is utilized.

\quad

\noindent {\bf Keywords:} binary contact path process, hydrodynamic limit, linear system, absolute continuity.

\section{Introduction}

In this paper we are concerned with the the binary contact path process on the lattice $\mathbb{Z}^d$ with $d\geq 3$. The binary contact path process $\{\eta_t\}_{t\geq 0}$ is a continuous-time Markov process with state space $[0,+\infty)^{\mathbb{Z}^d}$, i.e., at each vertex of $\mathbb{Z}^d$ there is a spin taking value in $[0,+\infty)$. To give the transition rates function of the process, we introduce some notations. For $x,y\in \mathbb{Z}^d$, we write $x\sim y$ when and only when $x$ and $y$ are neighbors, i.e., the $l_1$ norm of $x-y$ is $1$. We denote by $O$ the origin of $\mathbb{Z}^d$, i.e.,
\[
O=(0,0,\ldots,0).
\]
For any configuration $\eta\in [0,+\infty)^{\mathbb{Z}^d}$ and $x\in \mathbb{Z}^d$, we define $\eta^x\in [0,+\infty)^{\mathbb{Z}^d}$
as
\[
\eta^x(u)=
\begin{cases}
\eta(u) & \text{~if~}u\neq x,\\
0 & \text{~if~}u=x.
\end{cases}
\]
For any $\eta\in [0,+\infty)^{\mathbb{Z}^d}$ and $x,y\in \mathbb{Z}^d$ such that $x\sim y$, we define $\eta^{x,y}\in [0,+\infty)^{\mathbb{Z}^d}$ as
\[
\eta^{x,y}(u)=
\begin{cases}
\eta(u) & \text{~if~}u\neq x,\\
\eta(x)+\eta(y) & \text{~if~}u=x.
\end{cases}
\]
For each $x\in \mathbb{Z}^d$, we let $\{Y_x(t)\}_{t\geq 0}$ be a Poisson process with rate $1$. For any $x\sim y$, we let $\{U_{x,y}(t)\}_{t\geq 0}$ be a Poisson process with rate $\lambda>0$, where $\lambda$ is a constant called the infection rate. We assume that all these Poisson processes are independent. Note that we care about the order of $x$ and $y$, hence $U_{x,y}\neq U_{y,x}$. Then the binary contact path process evolves as follows. For any event moment $t$ of $Y_x(\cdot)$, $\eta_t=\eta_{t-}^x$. Note that $\eta_{t-}=\lim_{s<t,s\rightarrow t}\eta_s$, which is state of the process at the moment just before $t$. For any event moment $r$ of $U_{x,y}(\cdot)$, $\eta_r=\eta_{r-}^{x,y}$. For $0\leq t_1<t_2$, if there is no event moments of $Y_x(\cdot)$ or $\{U_{x,y}(\cdot):~y\sim x\}$ in $[t_1,t_2]$, then
\[
\eta_s(x)=\eta_{t_1}(x)\exp{\{(1-2\lambda d)(s-t_1)\}}
\]
 for any $t_1\leq s\leq t_2$, i.e.,
 \[
 \frac{d}{ds}\eta_s(x)=(1-2\lambda d)\eta_s(x)
 \]
 for $s\in [t_1,t_2]$.

 Intuitively, the binary contact path process describes the spread of an infection disease on $\mathbb{Z}^d$. $x$ is healthy if $\eta(x)=0$. $x$ is infected if $\eta(x)>0$ while $\eta(x)$ is the seriousness of the disease on $x$. $x$ is infected by a given neighbor $y$ at rate $\lambda$. When the infection occurs, the seriousness of the disease on $x$ is added  with that of $y$. When there is no infection occurs for $x$ during $[t_1,t_2]$, then $\eta(x)$ evolves according to the deterministic ODE $\frac{d}{dt}\eta_t(x)=(1-2\lambda d)\eta_t(x)$.

 The binary contact path process can be defined equivalently via its generator. According to the evolution of this process introduced above, the generator $\mathcal{L}$ of $\{\eta_t\}_{t\geq 0}$ is given by
 \begin{equation}\label{equ 1.1 generator}
     \mathcal{L}f(\eta)=\sum_{x\in \mathbb{Z}^d}\big[f(\eta^x)-f(\eta)\big]+\lambda\sum_{x\in \mathbb{Z}^d}\sum_{y\sim x}\big[f(\eta^{x,y})-f(\eta)\big]+(1-2\lambda d)\sum_{x\in \mathbb{Z}^d}f_x^{\prime}(\eta)\eta(x)
 \end{equation}
 for any $\eta\in [0,+\infty)^{\mathbb{Z}^d}$ and sufficiently smooth $f$, where $f_x^{\prime}(\eta)$ is the
partial derivative of $f$ with respect to the coordinate $\eta(x)$.

The binary contact path process is first introduced in \cite{Gri1983} as  an auxiliary model to study the critical value of the contact process, according to fact that the process $\{\xi_t\}_{t\geq 0}$ with state space $\{0,1\}^{\mathbb{Z}^d}$ defined as
\[
\xi_t(x)=
\begin{cases}
1 &\text{~if~}\eta_t(x)>0,\\
0 &\text{~if~}\eta_t(x)=0
\end{cases}
\]
for each $x\in \mathbb{Z}^d$ is a version of the basic contact process introduced in \cite{Har1974}. Utilizing the coupling relationship between $\{\eta_t\}_{t\geq 0}$ and $\{\xi_t\}_{t\geq 0}$ given above, it is proved in \cite{Gri1983} that the critical value of the contact process on $\mathbb{Z}^d$ with $d\geq 3$ is at most
\[
\frac{1}{2d(2\gamma_d-1)},
\]
 where $\gamma_d$ is the probability that the simple random walk on $\mathbb{Z}^d$ starting at $O$ never return to $O$ again.

The binary contact path process belongs to a class of continuous-time Markov processes called linear systems. For a detailed survey of the definition and main properties of linear systems, see Chapter 9 of \cite{Lig1985}.

In this paper we are concerned with the hydrodynamic limit of $\{\eta_t\}_{t\geq 0}$.  The history of hydrodynamics in probability theory goes back to the 1980s in \cite{Dobrushin&Siegmund-Schultze82,GalvesKipnisMarchioroPresutti81,Rost81}. The theory says that the microscopic density field of the concerned model, after properly space time scaling, is dominated macroscopically  be some PDE. We refer to \cite{DeMasiPresutti91,kipnis+landim99} for a comprehensive reading of the subject. The favorite model in this area, such as exclusion processes (cf. Chapter 4 in \cite{kipnis+landim99}) and zero range processes (cf. Chapter 5 in \cite{kipnis+landim99}), is mass conserved. However, the mass conserved property is absent in our model. The reason why we can consider the hydrodynamics of the binary contact path process lies in the fact that the average of mass is conserved, the same as the voter model considered in  \cite{PresuttiSpohn83}. Closely related to our work is  \cite{NagahataYoshida09}, where a central limit theorem for the density of particles in the context of binary context path processes was proved.

To give our main result, we define  $\{\eta^N_t\}_{t\geq 0}$ as the continuous-time Markov process with generator $N^2\mathcal{L}$ for each integer $N\geq 1$, where $\mathcal{L}$ is defined as in Equation \eqref{equ 1.1 generator}. That is to say, $\{\eta_t^N\}_{t\geq 0}$ is a version of $\{\eta_{tN^2}\}_{t\geq 0}$ with some initial condition $\eta^N_0\in [0,+\infty)^{\mathbb{Z}^d}$. For any $x\in \mathbb{R}^d$, we use $\delta_x(du)$ to denote the Dirac measure concentrated on $x$, i.e.,
\[
\delta_x(A)=
\begin{cases}
1 &\text{~if~}A\ni x,\\
0 & \text{~if~}A\not\ni x
\end{cases}
\]
for any Borel-measurable $A\subseteq \mathbb{R}^d$.  For any $t>0$, we use $\pi^N_t$ to denote the random empirical measure
\[
\pi^N_t := \frac{1}{N^d}\sum_{x\in \mathbb{Z}^d}\eta_t^N(x)\delta_{\frac{x}{N}}(du).
\]

We introduce the following notations. By $C_c (\mathbb{R}^d)$ denote the set of  continuous functions $f : \mathbb{R}^d \rightarrow \mathbb{R}$ with compact support. Let $\mathcal{M}_+ (\mathbb{R}^d)$ be the set of positive Radon measures in $\mathbb{R}^d$. For measures $\mu_n$, $n \geq 1$, and $\mu$ on $\mathcal{M}_+ (\mathbb{R}^d)$, say $\mu_n \rightarrow \mu$ in probability as $n \rightarrow \infty$ if for every test function $G \in C_c (\mathbb{R}^d)$, $<\mu_n,G> \rightarrow <\mu,G>$ in probability.

\begin{theorem} \label{Theorem 1.1 Main}
Let $\rho_0: \mathbb{R}^d \rightarrow [0,\infty)$ be bounded and integrable. Initially, $\eta^N_0 (x) = \rho_0 (x / N)$. Suppose $d \geq 3$ and
\[
\lambda>\frac{1}{2d(2\gamma_d-1)},
\]
then for all $t \geq 0$, as $N \rightarrow \infty$,
\begin{equation}\label{}
  \pi^N_t (d u) \rightarrow \rho (t, u) du~~\text{in probability,}
\end{equation}
where $\rho (t,u)$ is the unique solution of the heat equation
\begin{equation}\label{eq:intro1}
  \begin{cases}
    \partial_t \rho (t,u) = \lambda \Delta \rho (t,u), \\
    \rho (0,u) = \rho_0 (u).
  \end{cases}
\end{equation}
\end{theorem}

\proof[Remark 1] According to classic theory of heat equation, $\rho (t,u)$ has the following explicit form
\begin{equation}\label{eqn:intro2}
    \rho (t,u) = \int_{\mathbb{R}^d} (2 \pi t)^{- \frac{d}{2}} \exp \left\{ - \sum_{i=1}^{d} \frac{x_i^2}{2t}  \right\} \rho_0 (\sqrt{2 \lambda} x + u) d x.
\end{equation}

\qed

\proof[Remark 2]

The assumption $\lambda>\frac{1}{2d(2\gamma_d-1)}$ is needed according to some detailed techniques in our current proof. We have no idea whether  Theorem \ref{Theorem 1.1 Main} holds for smaller $\lambda$. We will work on this problem as a further study.

\qed

The paper is organized as follows. In Section \ref{section two} we prove an useful estimate for the convergence of the variance of empirical measure at a given macroscopic time $t$. Theorem \ref{Theorem 1.1 Main} is proved in Section \ref{section proof}. The main difficulty is to prove the absolute continuity of the limiting distribution, where the estimate proved in  Section \ref{section two} is used.

\section{Convergence of the Variance}\label{section two}
In this section we will prove the following lemma.
\begin{lemma}\label{lemma 2.1 variance}
Under the initial condition given in Theorem \ref{Theorem 1.1 Main}, for given $G\in C_c^2(\mathbb{R}^d)$ and $t>0$,
\[
\lim_{N\rightarrow+\infty}{\rm Var}\big(<\pi_t^N,G>\big)=0.
\]
\end{lemma}
With Lemma \ref{lemma 2.1 variance}, we can show that the limit of any sub-sequence of $\{\pi_t^N\}_{N\geq 1}$ is absolutely continuous with respect to the Lebesgue measure $du$. For details, see Section \ref{section proof}.

As a preparation of the proof of Lemma \ref{lemma 2.1 variance}, we introduce some notations and definitions. If $V$ is a countable set, then
\[
H:V\times V\rightarrow (0,+\infty)
\]
is called a $V\times V$ matrix. For two $V\times V$ matrices $H$ and $H_1$, $HH_1$ is defined in the same way as the product of two finite-dimensional matrices is defined. That is to say,
\[
HH_1(x,y)=\sum_{u\in V}H(x,u)H_1(u,y)
\]
conditioned on the sum is absolutely convergent for each pair of $(x,y)\in V\times V$. Similarly, for integer $n\geq 1$ and real number $t\geq 0$, we define
\[
H^2=HH, H^{n+1}=H^nH, e^{tH}=\sum_{n=0}^{+\infty}\frac{t^nH^n}{n!}
\]
conditioned on the sums concerned are all absolutely convergent.

We use $\{S_n\}_{n\geq 0}$ to denote the discrete-time simple random walk on $\mathbb{Z}^d$ such that
\[
P\big(S_{n+1}=y\big|S_{n}=x\big)=\frac{1}{2d}
\]
for any $n\geq 0, x\in \mathbb{Z}^d$ and $y\sim x$.
We use $\{X_t\}_{t\geq 0}$ to denote the continuous-time simple random walk on $\mathbb{Z}^d$ such that
\[
P\big(X_{t+s}=y\big|X_t=x\big)=\lambda s(1+o(1)) \text{~while~} P\big(X_{t+s}=x\big|X_t=x\big)=1-2d\lambda s(1+o(1))
\]
for any $t\geq 0, x\in \mathbb{Z}^d$ and $y\sim x$ as $s\rightarrow 0$.

For each $x\in \mathbb{Z}^d$, we define
\begin{equation}\label{equ 2.1 k}
k(x)=P\big(S_n=O\text{~for some~}n\geq 0\big|S_0=x\big).
\end{equation}
For any $t\geq 0$ and $x,y\in \mathbb{Z}^d$, we define
\begin{equation*}
p_t(x,y)=P\big(X_t=y\big|X_0=x\big).
\end{equation*}
We use $\mathbb{E}$ to denote the expectation operator throughout this paper, then we have the following two lemmas, which are crucial for us to prove Lemma \ref{lemma 2.1 variance}.
\begin{lemma}\label{lemma 2.2}
For any $t\geq 0$,
\[
\mathbb{E}\big(\eta_t(x)\big)=\sum_{y\in \mathbb{Z}^d}p_{t}(x,y)\mathbb{E}\big(\eta_0(y)\big).
\]
\end{lemma}

\begin{lemma}\label{lemma 2.3}
There exists a $(\mathbb{Z}^d\times \mathbb{Z}^d)\times(\mathbb{Z}^d\times \mathbb{Z}^d)$ matrix $M$ such that $\{e^{tM}\}_{t\geq 0}$ are well defined and
\[
\mathbb{E}\Big(\eta_t(x)\eta_t(y)\Big)=\sum_{u\in \mathbb{Z}^d}\sum_{v\in \mathbb{Z}^d}e^{tM}\Big((x,y),(u,v)\Big)\mathbb{E}\Big[\eta_0(u)\eta_0(v)\Big]
\]
for any $x,y\in \mathbb{Z}^d$, $t\geq 0$.
\end{lemma}

Readers familiar with the theory of linear systems can easily check that Lemmas \ref{lemma 2.2} and \ref{lemma 2.3} are direct applications of Theorems 9.1.27 and 9.3.1 in \cite{Lig1985}, which are extensions of Hille-Yosida Theorem for the linear system. For readers not familar with these two theorems, we put the proofs in the appendix.

Let $M$ be defined as in Lemma \ref{lemma 2.3}, then we write $M$ as $M_\lambda$ when we need to point out the infection rate $\lambda$. The following lemma is crucial for us to prove Lemma \ref{lemma 2.1 variance}.
\begin{lemma}\label{lemma 2.4}
For any $\lambda>\frac{1}{2d(2\gamma_d-1)}$, there exists $h_\lambda>0$ such that
\[
\sum_{u\in \mathbb{Z}^d}\sum_{v\in \mathbb{Z}^d}e^{tM_\lambda}\big((x,y),(u,v)\big)\leq \frac{k(y-x)+h_\lambda}{h_\lambda}
\]
for any $t\geq 0$ and $x,y\in \mathbb{Z}^d$, where $k(x)$ is defined as in Equation \eqref{equ 2.1 k}.
\end{lemma}
The proof of Lemma \ref{lemma 2.4} follows from the strategy introduced in Section 9.3 of \cite{Lig1985}, which we put in the appendix.

Now we give the proof of Lemma \ref{lemma 2.1 variance}.

\proof[Proof of Lemma \ref{lemma 2.1 variance}]

By Lemmas \ref{lemma 2.2} and \ref{lemma 2.3}, for $x,y\in \mathbb{Z}^d$,
\[
\mathbb{E}\big(\eta_t^N(x)\big)=\sum_{u\in \mathbb{Z}^d}p_{tN^2}(x,u)\rho_0(\frac{u}{N})
\]
and
\[
\mathbb{E}\big(\eta_t^N(x)\eta_tN(y)\big)=\sum_{u\in \mathbb{Z}^d}\sum_{v\in \mathbb{Z}^d}e^{tN^2M_\lambda}\big((x,y),(u,v)\big)\rho_0(\frac{u}{N})\rho_0(\frac{v}{N})
\]
under the ininital condition given in Theorem \ref{Theorem 1.1 Main}. Therefore,
\begin{align}\label{equ 2.2}
&{\rm Var}(<\pi_t^N, G>)={\rm Cov}\big(\frac{1}{N^d}\sum_{x\in \mathbb{Z}^d}G(\frac{x}{N})\eta_t^N(x),\frac{1}{N^d}\sum_{x\in \mathbb{Z}^d}G(\frac{x}{N})\eta_t^N(x)\big) \notag\\
&=\frac{1}{N^{2d}}\sum_{x\in \mathbb{Z}^d}\sum_{y\in \mathbb{Z}^d}G(\frac{x}{N})G(\frac{y}{N}){\rm Cov}\big(\eta_t^N(x),\eta_t^N(y)\big) \notag\\
&=\frac{1}{N^{2d}}\sum_{x\in \mathbb{Z}^d}\sum_{y\in \mathbb{Z}^d}G(\frac{x}{N})G(\frac{y}{N})\Big(\mathbb{E}\big(\eta_t^N(x)\eta_t^N(y)\big)-\mathbb{E}\big(\eta_t^N(x)\big)\mathbb{E}\big(\eta_t^N(y)\big)\Big)\\
&=\sum_{x\in \mathbb{Z}^d}\sum_{y\in \mathbb{Z}^d}G(\frac{x}{N})G(\frac{y}{N})
\sum_{u\in \mathbb{Z}^d}\sum_{v\in \mathbb{Z}^d}\frac{\rho_0(\frac{v}{N})\rho_0(\frac{u}{N})}{N^{2d}}\Big(e^{tN^2M_\lambda}\big((x,y),(u,v)\big)-p_{tN^2}(x,u)
p_{tN^2}(y,v)\Big).
\notag
\end{align}

We claim that
\begin{equation}\label{equ 2.3}
e^{tM_\lambda}\big((x,y),(u,v)\big)-p_{t}(x,u)
p_{t}(y,v)\geq 0
\end{equation}
for any $u,v,x,y\in \mathbb{Z}^d$ and $t\geq 0$. To check Equation \eqref{equ 2.3}, we need to utilize the explicit expression of $M_\lambda$. For details, see the appendix.
We use $\|\rho_0\|_{\infty}$ to denote $\sup\{\rho_0(x):x\in \mathbb{R}^d\}$, then by Equations \eqref{equ 2.2}, \eqref{equ 2.3} and the fact that
$\sum_{u\in \mathbb{Z}^d}p_t(x,u)=\sum_{v\in \mathbb{Z}^d}p_t(y,v)=1$,
\begin{align}\label{equ 2.4}
{\rm Var}(<\pi_t^N, G>)
\leq \frac{\|\rho_0\|^2_\infty}{N^{2d}}\sum_{x\in \mathbb{Z}^d}\sum_{y\in \mathbb{Z}^d}|G(\frac{x}{N})||G(\frac{y}{N})|
\Big[-1+\sum_{u,v\in \mathbb{Z}^d}e^{tN^2M_\lambda}\big((x,y),(u,v)\big)\Big].
\end{align}
When $\lambda>\frac{1}{2d(2\gamma_d-1)}$, by Equation \eqref{equ 2.4} and Lemma \ref{lemma 2.4},
\begin{equation}\label{equ 2.5}
{\rm Var}(<\pi_t^N, G>)\leq \frac{\|\rho_0\|^2_\infty}{N^{2d}h_\lambda}\sum_{x\in \mathbb{Z}^d}\sum_{y\in \mathbb{Z}^d}|G(\frac{x}{N})||G(\frac{y}{N})|k(y-x).
\end{equation}
For each $x\in \mathbb{Z}^d$, we use $\|x\|_{\infty}$ to denote the $l_{\infty}$ norm of $x$, then
\[
\lim_{\|x\|\rightarrow+\infty}k(x)=0
\]
according to the fact that the simple random on $\mathbb{Z}^d$ with $d\geq 3$ is transient. As a result, for any $\epsilon>0$, there exists $N_1(\epsilon)>1$ such that
\[
k(x)\leq \epsilon
\]
when $\|x\|_{\infty}>N_1(\epsilon)$. For $d\geq 3$ and $x\in \mathbb{Z}^d$,
\[
|\{y:\|y-x\|_\infty\leq N_1(\epsilon)\}|\leq [3N_1(\epsilon)]^d,
\]
where $|A|$ is the cardinality of the set $A$. Therefore, by Equation \eqref{equ 2.5},
\begin{align}\label{equ 2.6}
 {\rm Var}(<\pi_t^N, G>)\leq \frac{\|\rho_0\|^2_\infty}{N^{2d}h_\lambda}\sum_{x\in \mathbb{Z}^d}\sum_{y\in \mathbb{Z}^d}|G(\frac{x}{N})||G(\frac{y}{N})|\epsilon
 +\frac{\|\rho_0\|^2_\infty\|G\|_\infty}{N^{2d}h_\lambda}\sum_{x\in \mathbb{Z}^d}|G(\frac{x}{N})|\big[3N_1(\epsilon)\big]^d,
\end{align}
where $\|G\|_\infty=\sup\{|G(x)|:x\in \mathbb{R}^d\}$. For sufficiently large $N$,
\[
\frac{1}{N^{2d}}\sum_{x\in \mathbb{Z}^d}\sum_{y\in \mathbb{Z}^d}|G(\frac{x}{N})||G(\frac{y}{N})|
\leq 2\Big(\int_{\mathbb{R}^d}|G(u)|du\Big)^2
\]
while
\[
\frac{1}{N^d}\sum_{x\in \mathbb{Z}^d}|G(\frac{x}{N})|\leq 2\int_{\mathbb{R}^d}|G(u)|du.
\]
Therefore, by Equation \eqref{equ 2.6},
\[
{\rm Var}(<\pi_t^N, G>)\leq \frac{2\|\rho_0\|^2_\infty\epsilon}{h_\lambda}\Big(\int_{\mathbb{R}^d}|G(u)|du\Big)^2
+\frac{2\|\rho_0\|^2_\infty\|G\|_\infty\big[3N_1(\epsilon)\big]^d}{N^dh_\lambda}\int_{\mathbb{R}^d}|G(u)|du
\]
for sufficiently large $N$. Note that $N_1(\epsilon)$ only depends on $\epsilon$, hence
\begin{equation}\label{equ 2.7}
 \limsup_{N\rightarrow+\infty}{\rm Var}(<\pi_t^N, G>)\leq \frac{2\|\rho_0\|^2_\infty\epsilon}{h_\lambda}\Big(\int_{\mathbb{R}^d}|G(u)|du\Big)^2.
\end{equation}
Lemma \ref{lemma 2.1 variance} follows from Equation \eqref{equ 2.7} directly, since $\epsilon$ is arbitrary.

\qed

\section{Proof of the Main Theorem}\label{section proof}

In this section we prove Theorem \ref{Theorem 1.1 Main}. Fix $T > 0$. Let $Q^N$, $N \geq 1$, be the measure on $D\left( [0,T], \mathcal{M}_+ (\mathbb{R}^d) \right)$ induced by the process $\{ \pi^N_t, 0 \leq t \leq T \}$. The proof is divided into two steps. We first show in Lemma \ref{lemma:MT1} that the sequence of measures $\{Q^N, N \geq 1\}$ is tight. We then prove in Lemma \ref{lemma:MT2} that any weak limit of the sequence $\{Q^N, N \geq 1\}$ along some subsequence concentrates on the trajectory which is the solution of the heat equation \eqref{eq:intro1}.  The main difficulty is to prove that the limiting measure concentrates on absolutely continuous trajectories, where Lemma \ref{lemma 2.1 variance} is utilized.  Theorem \ref{Theorem 1.1 Main} follows from Lemmas \ref{lemma:MT1} and \ref{lemma:MT2} directly.

We first introduce two martingales, which are useful to prove Lemma \ref{lemma:MT1} and Lemma \ref{lemma:MT2}.  For each test function $G \in C_c^2 (\mathbb{R^d})$, by Dynkin's martingale formula,
\begin{equation}\label{eq:MT1}
M^N_t (G) := <\pi^N_t, G> - <\pi^N_0, G> - \int_0^t N^2 \mathcal{L} <\pi^N_s, G> d s
\end{equation}
and
\begin{equation}\label{eq:MT2}
A^N_t (G) := \big[ M^N_t (G) \big]^2 - \int_{0}^{t} d s \Big\{ N^2 \mathcal{L} <\pi^N_s, G>^2 - 2 <\pi^N_s, G> N^2 \mathcal{L} <\pi^N_s, G> \Big\}
\end{equation}
are both martingales. By direct calculation,
\begin{equation}\label{eq:MT3}
N^2 \mathcal{L} <\pi^N_s, G> = \frac{\lambda}{N^d} \sum_{x \in \mathbb{Z}^d} \eta^N_s (x) \Delta_N G (\frac{
x}{N}),
\end{equation}
where $\Delta_N$ is the discrete Laplace, $\Delta_N G (x / N) = N^2 \sum_{|y - x| = 1} [G (y / N) - G (x / N)]$ and
\begin{equation}\label{eq:MT4}
\begin{split}
   N^2 \mathcal{L} <\pi^N_t, G>^2 &- 2 <\pi^N_t, G> N^2 \mathcal{L} <\pi^N_t, G> \\
     & = N^{2 - 2 d} \sum_{x \in \mathbb{Z}^d} \left\{  \eta_s^N (x)^2 \left[ G (\frac{x}{N})^2 + \lambda \sum_{|y - x|=1} G (\frac{y}{N})^2 \right] \right\}.
\end{split}
\end{equation}

\begin{lemma}[Tightness]\label{lemma:MT1}
The sequence $\{Q^N, N \geq 1\}$ is tight.
\end{lemma}

\begin{proof}
It is well known that the sequence $\{Q^N, N \geq 1\}$ is tight if and only if the sequence $\{ Q^N G^{-1}, N \geq 1\}$ is tight for each $G \in C_c^2 (\mathbb{R^d})$, where $Q^N G^{-1}$ is the measure on $D\left( [0,T], \mathbb{R} \right)$ induced by  the process $\{ <\pi^N_t, G>, 0 \leq t \leq T\}$. By Aldous criteria for tightness in $D\left( [0,T], \mathbb{R} \right)$, it suffices to check the following two conditions:\\
(i) For all $0 \leq t \leq T$,
$$
\lim_{M \rightarrow \infty} \limsup_{N \rightarrow \infty} Q^N G^{-1} (|\omega_t| > M) = 0.
$$
(ii) Denote by $\mathfrak{T}$ the set of stopping times with respect to the natural filtration bounded by $T$. For all $\delta > 0$,
$$
\lim_{\gamma \rightarrow 0} \limsup_{N \rightarrow \infty} \sup_{\tau \in \mathfrak{T}, \theta \leq \gamma} Q^N G^{-1} (|\omega_\tau - \omega_{\tau + \theta}| > \delta) = 0.
$$

To check condition (i), for all $0 \leq t \leq T$,
\begin{equation}\label{eq:MT7}
\begin{split}
   \limsup_{N \rightarrow \infty} Q^N G^{-1} (|\omega_t| > M) &= \mathbb{P} (|<\pi^N_t, G>| >  M) \\
     & \leq \frac{1}{M} \limsup_{N \rightarrow \infty} \mathbb{E} \left[ \frac{1}{N^d} \sum_{x \in \mathbb{Z}^d} \eta^N_t (x) |G (\frac{x}{N})| \right].
\end{split}
\end{equation}
Since $G$ is bounded and by Lemma \ref{lemma 2.2},
\begin{equation}\label{eq:MT10}
  \mathbb{E} \left[ \sum_x \eta^N_t (x) \right] =  \mathbb{E} \left[ \sum_x \eta^N_0 (x) \right],
\end{equation}
the right-hand side of \eqref{eq:MT7} is bounded by
\begin{equation*}
\frac{||G||_\infty}{M} \limsup_{N \rightarrow \infty} \frac{1}{N^d} \sum_x \rho_0 (x / N),
\end{equation*}
which is equal to
$$
\frac{||G||_\infty}{M} \int_{\mathbb{R}^d} \rho_0 (u) d u.
$$
Therefore, condition (i) is astisfied by the integrability of $\rho_0$.

To check condition (ii), by equations \eqref{eq:MT1} and \eqref{eq:MT3}, we deal with
$$
\int_\tau^{\tau + \theta} \frac{\lambda}{N^d} \sum_{x \in \mathbb{Z}^d} \eta^N_s (x) \Delta_N G (\frac{
x}{N}) d s
$$
and
$$
M^N_{\tau + \theta} (G) - M^N_\tau (G)
$$
separately.  For the first term, by Chebyshev's inequality,
\begin{equation}\label{}
\begin{split}
  \mathbb{P} \left( \left| \int_\tau^{\tau + \theta} \frac{\lambda}{N^d} \sum_{x \in \mathbb{Z}^d} \eta^N_s (x) \Delta_N G (\frac{x}{N}) d s \right| > \delta \right)  &\leq  \lambda ||\Delta G||_\infty \delta^{-1} \mathbb{E} \left[ \int_\tau^{\tau + \theta} N^{-d} \sum_{x \in \mathbb{Z}^d} \eta^N_s (x) d s \right ]\\
  & =  \lambda ||\Delta G||_\infty \delta^{-1}  \int_0^{\theta} N^{-d} \mathbb{E}  \left[ \sum_{x \in \mathbb{Z}^d} \eta^N_{s+\tau} (x) \right] ds\\
  & =  \lambda ||\Delta G||_\infty \theta \delta^{-1} N^{-d} \sum_{x \in \mathbb{Z}^d} \rho_0 (x / N).
\end{split}
\end{equation}
The last equality is due to the fact that $\big\{ \sum_{x \in \mathbb{Z}^d} \eta^N_t (x) \big\}_{t \geq 0}$ is a martingale, which is implied by equation \eqref{eq:MT10}, and that $\tau$ is a bounded stopping time. For the second term,
\begin{equation}\label{}
  \mathbb{P} (|M^N_{\tau + \theta} (G) - M^N_\tau (G)| > \delta) \leq \delta^{-2} \mathbb{E} [ (M^N_{\tau + \theta} (G) - M^N_\tau (G))^2 ].
\end{equation}
By equations \eqref{eq:MT2} and \eqref{eq:MT4}, the last term is bounded by
\begin{equation}\label{eq:MT5}
 \delta^{-2} N^{2 - 2 d} \int_0^T \sum_{x \in \mathbb{Z}^d} \mathbb{E} [ \eta^N_s (x) ^2] \left[G (\frac{x}{N})^2 + \lambda \sum_{|y - x|=1} G (\frac{y}{N})^2 \right] d s.
\end{equation}
By the boundedness of the initial density profile $\rho_0$, Lemma \ref{lemma 2.3} and Lemma \ref{lemma 2.4},
\begin{equation}\label{eq:MT11}
\sup_{N,s,x} \mathbb{E} [ \eta^N_s (x) ^2] \leq ||\rho_0||_\infty^2 \frac{k (O) + h_\lambda}{h_\lambda} < \infty.
\end{equation}
Note that $G$ has compact support, hence there are at most $C N^d$ nonzero summation  terms in \eqref{eq:MT5} for some finite $C > 0$. By the boundedness of $G$ and \eqref{eq:MT11}, formula \eqref{eq:MT5} is bounded by $C T \delta^{-2} N^{2 - d}$ for some finite constant $C > 0$, and the result follows.
\end{proof}

\begin{lemma}[Uniqueness of the limit point]\label{lemma:MT2}
  If $Q$ is any weak limit of the sequence $\{Q^N, N \geq 1\}$ along some subsequence, then the measure $Q$ concentrates on the trajectory  which is the solution of the heat equation \eqref{eq:intro1}.
\end{lemma}

We first recall the definition of the weak solution to the heat equation. Fix an initial bounded density profile $\rho_0: \mathbb{R}^d \rightarrow \mathbb{R}_+$. A bounded function $\rho: \mathbb{R}_+ \times \mathbb{R}^d \rightarrow \mathbb{R}$ is a weak solution of the heat equation \eqref{eq:intro1} if
\begin{equation}
    \int_{\mathbb{R}^d} \rho (t,u) G (u) du = \int_{\mathbb{R}^d} \rho_0 (u) G (u) d u + \int _ { 0 } ^ { t } d s \int_{\mathbb{R}^d} \rho (s,u) G (u) du
\end{equation}
for all $t \geq 0$ and for all $G \in C_c (\mathbb{R}^d)$. It is well known that the weak solution to the heat equation is unique, and therefore is given by  \eqref{eqn:intro2}.

\begin{proof}[Proof of Lemma \ref{lemma:MT2}]
Assume $\{ \pi_t^{N_k}, 0 \leq t \leq T \}$ converges in distribution to $\{ \pi_t, 0 \leq t \leq T \}$, whose law is $Q$. The lemma is a direct consequence of the following two statements: (i)  For each $0 \leq t \leq T$,  $\mathbb{E} [\pi_t (d u)] = \rho (t,u) d u$, where $\rho (t,u)$ is the solution of the heat equation \eqref{eq:intro1}.  (ii) $Q (\pi_t = \mathbb{E} [\pi_t]~~\text{for all $0 \leq t \leq T$}) = 1$.

For the first statement, we first show that for all $G \in C_c^2 (\mathbb{R}^d)$, with probability one, for all $0 \leq t \leq T$,
\begin{equation}\label{eq:MT8}
<\pi_t, G>  = <\pi_0, G> + \lambda \int_{0}^{t} <\pi_s, \Delta G> d s.
\end{equation}
Note that the mapping from the trajectory $\{ \pi_t, 0 \leq t \leq T\}$ to
$$
\sup_{0 \leq t \leq T} |<\pi_t, G> - <\pi_0, G> - \lambda \int_{0}^{t} <\pi_s, \Delta G> d s|
$$ is continuous. Therefore, for any $\epsilon > 0$,
\begin{equation}\label{eq:MT6}
\begin{split}
  &Q \left( \sup_{0 \leq t \leq T} |<\pi_t, G> - <\pi_0, G> - \lambda \int_{0}^{t} <\pi_s, \Delta G> d s| > \epsilon \right) \\
  &\leq \liminf_{k \rightarrow \infty} Q^{N_k} \left( \sup_{0 \leq t \leq T} |<\pi_t, G> - <\pi_0, G> - \lambda \int_{0}^{t} <\pi_s, \Delta G> d s|  > \epsilon \right )
\end{split}
\end{equation}
since the event estimated above is open. By equation \eqref{eq:MT1}, the right-hand side of the above inequality is equal to
\begin{equation}
\liminf_{k \rightarrow \infty} \mathbb{P}  \big( \sup_{0 \leq t \leq T} |M_t^{N_k} (G) + \lambda \int_{0}^{t} <\pi_s^{N_k}, \Delta_{N_k} G - \Delta G> d s|  > \epsilon \big),
\end{equation}
which is bounded above by
\begin{equation}\label{eq:MT9}
   \liminf_{k \rightarrow \infty} \mathbb{P}  \big( \sup_{0 \leq t \leq T} |M_t^{N_k} (G) |  > \epsilon/2 \big)
  \leq 4 \epsilon^{-2} \liminf_{k \rightarrow \infty}  \mathbb{E}  [M_T^{N_k} (G)^2]
\end{equation}
since there exists a constant $C > 0$, such that for all $0 \leq t \leq T$ and large enough $N$,
\begin{equation}
 \mathbb{E} \big[ |<\pi_s^N, \Delta G> - <\pi_s^N, \Delta_N G> | \big] \leq \frac{C}{N}.
\end{equation}
By the proof in Lemma \ref{lemma:MT1}, $\mathbb{E}  [M_T^{N} (G)^2]$ is bounded above by $C N^{2 - d}$ for some finite positive constant $C$, and it follows that the right-hand side of  \eqref{eq:MT9} is zero. Therefore,
\begin{equation}
  Q \left( \sup_{0 \leq t \leq T} |<\pi_t, G> - <\pi_0, G> - \lambda \int_{0}^{t} <\pi_s, \Delta G> d s| > \epsilon \right) = 0
\end{equation}
for all $\epsilon > 0$, from which \eqref{eq:MT8} follows directly.

Next we show that $\mathbb{E} [\pi_t(du)]$ is absolutely continuous. Note that for any $G \in C_c^2 (\mathbb{R^d})$ and for any $t \in [0,T]$, the coordinate mapping from $\{ \omega_t, 0 \leq t \leq T \}$ to $\omega_t$ is almost surely continuous under $Q G^{-1}$ by formula \eqref{eq:MT8}. Therefore, as $k \rightarrow \infty$,
\begin{equation}
<\pi_t^{N_k}, G> \rightarrow <\pi_t,G> ~~\text{in distribution}.
\end{equation}
By Lemma \ref{lemma 2.1 variance} and the estimates for the right-hand side of \eqref{eq:MT7},
\begin{equation}
\mathbb{E} [ <\pi_t^{N}, G>^2 ] = \rm{Var} ( <\pi_t^{N}, G> ) +  \mathbb{E} [ <\pi_t^{N}, G> ]^2 < \infty
\end{equation}
uniformly in $N$, which implies the uniform integrability of the sequence $\mathbb{E}  [ <\pi_t^{N_k}, G> ]$. Thus,
\begin{equation}\label{}
   \mathbb{E}  [ <\pi_t^{N_k}, G> ] \rightarrow  \mathbb{E} [ <\pi_t,G> ].
\end{equation}
By Lemma \ref{lemma 2.2} and the boundedness of the initial density profile $\rho_0$, $\sup_{t,x,N} \mathbb{E} [\eta^N_t (x)] \leq || \rho_0 ||_\infty$. Therefore,
\begin{equation}
\begin{split}
\mathbb{E} [ <\pi_t,G> ] &= \lim_{k \rightarrow \infty} \mathbb{E}  [ <\pi_t^{N_k}, G> ] \\
&\leq || \rho_0 ||_\infty \lim_{k \rightarrow \infty} \frac{1}{N^d} \sum_{x \in \mathbb{Z}^d} G (\frac{x}{N}) = || \rho_0 ||_\infty \int G (u) d u,
\end{split}
\end{equation}
which implies the absolutely continuous of $\mathbb{E} [\pi_t(du)]$.

Denote by $\rho (t,u)$ the Radon-Nikodym derivative of the measure $\mathbb{E} [\pi_t (d u)]$. Note that
\begin{equation}\label{}
  <\pi_0,G> = \lim_{N \rightarrow \infty}  <\pi_0^N,G> = \sum_{x} G (\frac{x}{N}) \rho_0 (\frac{x}{N}) = \int \rho_0 (u) G (u) du.
\end{equation}
Taking expectation in both hands of equation \eqref{eq:MT8},
\begin{equation}
    \int_{\mathbb{R}^d} \rho (t,u) G (u) du = \int_{\mathbb{R}^d} \rho_0 (u) G (u) d u + \int _ { 0 } ^ { t } d s \int_{\mathbb{R}^d} \rho (s,u) G (u) du,
\end{equation}
which shows that $\rho (t,u)$ is the weak solution of the heat equation \eqref{eq:intro1}, and hence the unique solution.

For the second statement, by Fatou's Lemma,
\begin{equation}
\mathbb{E} [ <\pi_t,G>^2 ] \leq \liminf_{N \rightarrow \infty} \mathbb{E}  [ <\pi_t^N, G>^2 ].
\end{equation}
Therefore, for all $G \in C_c^2 (\mathbb{R^d})$,
\begin{equation}\label{}
\begin{split}
 \rm{Var} ( <\pi_t,G> ) & = \mathbb{E} [ <\pi_t,G>^2 ] - (\mathbb{E} [ <\pi_t,G> ])^2\\
 & \leq \liminf_{k \rightarrow \infty} \{ \mathbb{E}  [ <\pi_t^{N_k}, G>^2 ] - (\mathbb{E}  [ <\pi_t^{N_k}, G> ])^2 \}\\
 & = \liminf_{k \rightarrow \infty} \rm{Var}( <\pi_t^{N_k}, G> ) = 0
\end{split}
\end{equation}
according to Lemma \ref{lemma 2.1 variance}.
As a result, for all $0 \leq t \leq T$, $Q (\pi_t = \mathbb{E} [ \pi_t] ) = 1$. Since the trajectories are right continuous and have left limits, $Q (\pi_t = \mathbb{E}[ \pi_t], \forall 0 \leq t \leq T) = 1$.

\end{proof}

\begin{proof}[Proof of Theorem \ref{Theorem 1.1 Main}]
Let $d_0$ be the metric on $D ([0,T], \mathcal{M}_+ (\mathbb{R}^d))$ which induces  the Skorohod topology and makes the space complete and separable. Suppose $Q^N$ does not converge weakly to the measure $\{ \rho (t,u) du, 0 \leq t \leq T\} =: Q$. Then there exist $\epsilon > 0$ and some subsequence $Q^{N_k}$ such that $d_0 (Q^{N_k}, Q) > \epsilon$ for all $k$.  By Lemma \ref{lemma:MT1}, the sequence  $Q^{N_k}$ has a subsequence $Q^{N_{k_l}}$ such that $Q^{N_{k_l}}$ converges, and denote the limit by $\hat{Q}$. By lemma \ref{lemma:MT2}, $Q = \hat{Q}$, which is a contradiction. Therefore,  $Q^N$ converges weakly to the measure $\{ \rho (t,u) du, 0 \leq t \leq T\}$.

Since the measure $\rho (t,u) du$ is vaguely continuous in $t$, $\pi_t^N$ converges in distribution to the deterministic measure $\rho (t,u) du$.  Since the limiting measure is deterministic, the convergence is in probability.
\end{proof}

\textbf{Acknowledgments.} The authors are grateful to the financial support from the National Natural Science Foundation of China with
grant numbers 11501542 and 11531001.

\appendix{}
\section{Some properties of the binary contact path process}
In this appendix we give the proofs of Lemmas \ref{lemma 2.2}, \ref{lemma 2.3}, \ref{lemma 2.4} and Equation \eqref{equ 2.3}. Our proofs are based on the theory of the linear system introduced in Chapter 9 of \cite{Lig1985}.

\proof[Proof of Lemma \ref{lemma 2.2}]

Let $\{S(t)\}_{t\geq 0}$ be the semi-group of $\{\eta_t\}_{t\geq 0}$, then Theorem 9.1.27 of \cite{Lig1985} shows that
\[
\frac{d}{dt}S(t)f(\eta)=S(t)\mathcal{L}f(\eta)
\]
holds for $f$ with the form $f(\eta)=\eta(x)$, which can be considered as an extension of Hille-Yosida Theorem for the linear system.
Then, for each $x\in \mathbb{Z}^d$,
\begin{align}\label{equ A.1}
\frac{d}{dt}\mathbb{E}\big(\eta_t(x)\big)&=\mathbb{E}\big(0-\eta_t(x)\big)+\lambda\sum_{y\sim x}\mathbb{E}\big(\eta_t(x)+\eta_t(y)-\eta_t(x)\big)+(1-2\lambda d)
\mathbb{E}\big(\eta_t(x)\big) \notag\\
&=-2d\lambda \mathbb{E}\big(\eta_t(x)\big)+\lambda\sum_{y\sim x}\mathbb{E}\big(\eta_t(y)\big).
\end{align}
For any $t\geq 0$, we define $F_t:\mathbb{Z}^d\rightarrow \mathbb{R}$ as
\[
F_t(x)=\mathbb{E}\big(\eta_t(x)\big)
\]
for each $x\in Z^d$. For a countable set $V$,  $V\times V$ matrix $H$ and $K:V\rightarrow\mathbb{R}$, we define $HK:V\rightarrow\mathbb{R}$ as
\[
HK(x)=\sum_{y\in V}H(x,y)K(y)
\]
for each $x\in V$ conditioned on the sum is absolutely convergent. By Equation \eqref{equ A.1},
\[
\frac{d}{dt}F_t=H_2F_t
\]
for any $t\geq 0$, where
\[
H_2(x,y)=
\begin{cases}
-2d\lambda & \text{~if~}y=x,\\
\lambda & \text{~if~}y\sim x,\\
0 & \text{~else},
\end{cases}
\]
i.e., $H_2$ is the Q-matrix of the simple random walk $\{X_t\}_{t\geq 0}$ introduced in Section \ref{section two}. According to the theory of the linear ordinary differential equation on a Banach space introduced in \cite{Deimling1977},
\[
F_t=e^{tH_2}F_0
\]
for any $t\geq 0$ and hence
\begin{equation}\label{equ A.2}
\mathbb{E}\big(\eta_t(x)\big)=\sum_{y\in\mathbb{Z}^d}e^{tH_2}(x,y)\mathbb{E}\big(\eta_0(y)\big)
\end{equation}
for each $x\in \mathbb{Z}^d$. Since $H_2$ is the Q-matrix of $\{X_t\}_{t\geq 0}$,
\begin{equation}\label{equ A.3}
e^{tH_2}(x,y)=p_t(x,y)
\end{equation}
for any $x,y\in \mathbb{Z}^d$. Lemma \ref{lemma 2.2} follows from Equations \eqref{equ A.2} and \eqref{equ A.3} directly.

\qed

\proof[Proof of Lemma \ref{lemma 2.3}]

According to Theorem 9.3.1 of \cite{Lig1985}, which is also an extension of Hille-Yosida Theorem,
\[
\frac{d}{dt}S(t)f(\eta)=S(t)\mathcal{L}f(\eta)
\]
for $f$ with the form $f(\eta)=\eta(x)\eta(y)$, where $\{S(t)\}_{t\geq 0}$ is the semi-group of the bianry contact path process as we have introduced. Therefore,
\begin{align}\label{equ A.4}
\frac{d}{dt}\mathbb{E}\big(\eta_t^2(x)\big)&=\mathbb{E}\big(0-\eta_t^2(x)\big)+\lambda\sum_{y\sim x}\mathbb{E}\Big[\big(\eta_t(x)+\eta_t(y)\big)^2-\eta_t^2(x)\Big]
+2(1-2\lambda d)\mathbb{E}\big(\eta_t^2(x)\big)\notag\\
&=(1-4\lambda d)\mathbb{E}\big(\eta_t^2(x)\big)+\lambda\sum_{y\sim x}\mathbb{E}\big(\eta_t^2(y)\big)+2\lambda\sum_{y\sim x}\mathbb{E}\big(\eta_t(x)\eta_t(y)\big)
\end{align}
and
\begin{equation}\label{equ A.5}
\frac{d}{dt}\mathbb{E}\big(\eta_t(x)\eta_t(y)\big)=-4\lambda d\mathbb{E}\big(\eta_t(x)\eta_t(y)\big)+\lambda\sum_{u\sim x}\mathbb{E}\big(\eta_t(u)\eta_t(y)\big)
+\lambda\sum_{v\sim y}\mathbb{E}\big(\eta_t(x)\eta_t(v)\big)
\end{equation}
for $x\neq y$. For any $t\geq 0$, we define $\Gamma_t:\mathbb{Z}^d\times \mathbb{Z}^d\rightarrow \mathbb{R}$ as
\begin{equation*}
\Gamma_t(x,y)=\mathbb{E}\big(\eta_t(x)\eta_t(y)\big)
\end{equation*}
for any $x,y\in \mathbb{Z}^d$. Then, by Equations \eqref{equ A.4} and \eqref{equ A.5},
\[
\frac{d}{dt}\Gamma_t=M_\lambda \Gamma_t,
\]
where $M_\lambda$ is a $(\mathbb{Z}^d\times \mathbb{Z}^d)\times(\mathbb{Z}^d\times \mathbb{Z}^d)$ matrix such that
\[
M_\lambda\big((x,y),(u,v)\big)=
\begin{cases}
(1-4\lambda d) & \text{~if~}x=y=u=v,\\
\lambda & \text{~if~}x=y, u\sim x \text{~and~}v=u,\\
\lambda & \text{~if~}y=x=u, v\sim x,\\
\lambda & \text{~if~}y=x=v, u\sim x,\\
-4\lambda d & \text{~if~}y\neq x, u=x \text{~and~}v=y,\\
\lambda & \text{~if~}y\neq x, u\sim x \text{~and~}v=y,\\
\lambda & \text{~if~}y\neq x, u=x \text{~and~}v\sim y,\\
0 & \text{~else}
\end{cases}
\]
for any $x,y,u,v\in \mathbb{Z}^d$. As a result, according to the theory of the linear ordinary differential equation on a Banach space introduced in \cite{Deimling1977},
\[
\Gamma_t=e^{tM_\lambda}\Gamma_0
\]
for any $t\geq 0$, Lemma \ref{lemma 2.3} follows from which directly.

\qed

\proof[Proof of Lemma \ref{lemma 2.4}]

We use $\{\widehat{\eta}_t\}_{t\geq 0}$ to denote the binary contact path process with intial condition
\[
\widehat{\eta}_t(x)=1
\]
for all $x\in \mathbb{Z}^d$. By Lemma \ref{lemma 2.3},
\begin{equation}\label{equ A.6}
 \sum_{u\in \mathbb{Z}^d}\sum_{v\in \mathbb{Z}^d}e^{tM_\lambda}\big((x,y),(u,v)\big)=\mathbb{E}\big(\widehat{\eta}_t(x)\widehat{\eta}_t(y)\big)
\end{equation}
for any $x,y\in \mathbb{Z}^d$. For any $x\in \mathbb{Z}^d$, we define
\[
J_t (x)=\mathbb{E}\big(\widehat{\eta}_t(O)\widehat{\eta}_t(x)\big).
\]
Under the initial condition $\widehat{\eta}_0(x)=1$ for all $x\in \mathbb{Z}^d$, the binary contact path process is spatial homogeneous. Therefore,
\[
J_t(x)=\mathbb{E}\big(\widehat{\eta}_t(y)\widehat{\eta}_t(x+y)\big)
\]
for any $y\in \mathbb{Z}^d$ and
\begin{equation}\label{equ A.7}
 \sum_{u\in \mathbb{Z}^d}\sum_{v\in \mathbb{Z}^d}e^{tM_\lambda}\big((x,y),(u,v)\big)=J_t(y-x)
\end{equation}
according to Equation \eqref{equ A.6}. According to a similar analysis with that leading to Equations \eqref{equ A.4} and \eqref{equ A.5},
\[
\frac{d}{dt}J_t(x)=-4\lambda dJ_t(x)+2\lambda\sum_{y\sim x}J_t(y)
\]
for any $x\neq O$ while
\[
\frac{d}{dt}J_t(O)=(1-2\lambda d)J(O)+4\lambda dJ_t(e_1),
\]
where $e_1=(1,0,\ldots,0)$, i.e., the first elementary unit vector of $\mathbb{Z}^d$. Therefore,
\[
J_t=e^{t\Psi}J_0,
\]
where
\[
\Psi(x,y)=
\begin{cases}
-4\lambda d & \text{~if~}x\neq O \text{~and~}y=x,\\
2\lambda & \text{~if~}x\neq O \text{~and~}y\sim x,\\
(1-2\lambda d) & \text{~if~} x=y=O,\\
4\lambda d & \text{~if~}x=O \text{~and~}y=e_1,\\
0 & \text{~else}.
\end{cases}
\]
As a result, according to the fact that $J_0(y)\equiv 1$,
\begin{equation}\label{equ A.8}
J_t(x)=\sum_{y\in \mathbb{Z}^d}e^{t\Psi}(x,y)J_0(y)=\sum_{y\in \mathbb{Z}^d}e^{t\Psi}(x,y)
\end{equation}
for each $x\in \mathbb{Z}^d$. When $\lambda>\frac{1}{2d(2\gamma_d-1)}$, we define
\[
h_\lambda=\frac{2\lambda d(2\gamma_d-1)-1}{1+2d\lambda}>0.
\]
For each $x\in \mathbb{Z}^d$, let $\Lambda(x)=k(x)+h_\lambda$, then
\[
\Psi \Lambda=0 \text{\quad}(\text{zero function on $\mathbb{Z}^d$})
\]
according to the facts that $\gamma_d=1-k(e_1)$ and $k(x)=\frac{1}{2d}\sum_{y\sim x}k(y)$ for any $x\neq O$. Therefore, $\Lambda$ is the eigenvector of $\Psi$ with respect to the eigenvalue $0$. Then,
\[
e^{t\Psi}\Lambda=\Lambda+\sum_{n=1}^{+\infty}\frac{t^n\Psi^{n-1}\Psi \Lambda}{n!}=\Lambda
\]
for any $t\geq 0$, i.e., $\Lambda$ is the eigenvector of $e^{t\Psi}$ with respect to the eigenvalue $e^{t0}=1$ for any $t\geq 0$. As a result,
\begin{equation}\label{equ A.9}
\Lambda(x)=\sum_{y\in \mathbb{Z}^d}e^{t\Psi}(x,y)\Lambda(y)
\end{equation}
for any $t\geq 0$ and $x\in \mathbb{Z}^d$. Since $\Psi(x,y)\geq 0$ when $x\neq y$,
\[
e^{t\Psi}(x,y)\geq 0
\]
for any $x,y\in \mathbb{Z}^d$. Therefore, according to Equations \eqref{equ A.8}, \eqref{equ A.9} and the fact that $\inf_{x\in \mathbb{Z}^d} \Lambda(x)=h_\lambda$,
\begin{align}\label{equ A.10}
J_t(x)=\sum_{y\in \mathbb{Z}^d}e^{t\Psi}(x,y)\leq \sum_{y\in \mathbb{Z}^d}e^{t\Psi}(x,y)\frac{\Lambda(y)}{h_\lambda}=\frac{\Lambda(x)}{h_\lambda}
\end{align}
for any $x\in \mathbb{Z}^d$. Lemma \ref{lemma 2.4} follows from Equations \eqref{equ A.7} and \eqref{equ A.10} directly.

\qed

\proof[Proof of Equation \eqref{equ 2.3}]

Let $\{X_t\}_{t\geq 0}$ be the continuous-time simple random walk defined as in Section \ref{section two} and $\{Y_t\}_{t\geq 0}$ be an independent copy of $\{X_t\}_{t\geq 0}$, then
\[
p_t(x,u)p_t(y,v)=e^{tC_\lambda}\big((x,y),(u,v)\big),
\]
where $C_\lambda$ is the Q-matrix of $\{(X_t,Y_t)\}_{t\geq 0}$, i.e.,
\[
C_\lambda\big((x,y),(u,v)\big)=
\begin{cases}
-4\lambda d &\text{~if~}u=x\text{~and~}v=y,\\
\lambda & \text{~if~}u=x\text{~and~}v\sim y,\\
\lambda & \text{~if~}u\sim x\text{~and~}v=y,\\
0 & \text{~else}.
\end{cases}
\]
According to the expression of $M_\lambda$,
\begin{equation}\label{equ A.11}
M_\lambda\big((x,y),(u,v)\big)\geq C_\lambda\big((x,y),(u,v)\big)
\end{equation}
for any $x,y,u,v\in \mathbb{Z}^d$. Let $I$ be the $(\mathbb{Z}^d\times \mathbb{Z^d})\times (\mathbb{Z}^d\times \mathbb{Z}^d)$ identity matrix, then
\[
(4\lambda dI+M_\lambda)\big((x,y),(u,v)\big)\geq 0
\]
and
\[
(4\lambda dI+C_\lambda)\big((x,y),(u,v)\big)\geq 0
\]
for any $x,y,u,v\in \mathbb{Z}^d$.

Then, by Equation \eqref{equ A.11},
\[
e^{t(4\lambda dI+M_\lambda)}\big((x,y),(u,v)\big)\geq e^{t(4\lambda dI+C_\lambda)}\big((x,y),(u,v)\big)
\]
for any $x,y,u,v\in \mathbb{Z}^d$ and $t\geq 0$, Equation \eqref{lemma 2.3} follows from which by canceling $e^{4\lambda d t}$ in both sides.

\qed

\end{document}